\documentclass[a4paper]{amsart}
\usepackage[english]{babel}
\usepackage{amsmath,amsthm,amssymb,amsfonts}

\DeclareMathOperator{\Kern}{\mathcal{K}}
\DeclareMathOperator{\supp}{\mathrm{supp}}
\DeclareMathOperator{\tr}{\mathrm{tr}}

\newcommand{\Sz}{\mathcal{S}}

\newcommand{\N}{\mathbb{N}}
\newcommand{\Z}{\mathbb{Z}}
\newcommand{\R}{\mathbb{R}}
\newcommand{\C}{\mathbb{C}}

\newcommand{\leftopenint}{\left]}
\newcommand{\rightopenint}{\right[}
\newcommand{\leftclosedint}{\left[}
\newcommand{\rightclosedint}{\right]}

\newtheorem{thm}{Theorem}
\newtheorem{prp}[thm]{Proposition}

\newtheorem{lem}[thm]{Lemma}

\newcommand{\vecY}{\mathbf{Y}}

\newcommand{\Ell}{\mathcal{L}}

\begin{document}
\title[Spectral multipliers on $N_{3,2}$]{$L^p$ spectral multipliers on the free group $N_{3,2}$}
\author{Alessio Martini}
\address{Alessio Martini \\ Mathematisches Seminar \\ Christian-Albrechts-Universit\"at zu Kiel \\ Ludewig-Meyn-Str.\ 4 \\ D-24118 Kiel \\ Germany}
\email{martini@math.uni-kiel.de}
\author{Detlef M\"uller}
\address{Detlef M\"uller \\ Mathematisches Seminar \\ Christian-Albrechts-Universit\"at zu Kiel \\ Ludewig-Meyn-Str.\ 4 \\ D-24118 Kiel \\ Germany}
\email{mueller@math.uni-kiel.de}

\thanks{The first-named author gratefully acknowledges the support of the Alexander von Humboldt Foundation.}

\begin{abstract}
Let $L$ be the homogeneous sublaplacian on the $6$-dimensional free $2$-step nilpotent Lie group $N_{3,2}$ on $3$ generators. We prove a theorem of Mihlin-H\"ormander type for the functional calculus of $L$, where the order of differentiability $s > 6/2$ is required on the multiplier.
\end{abstract}

\maketitle

\section{Introduction}

The free $2$-step nilpotent Lie group $N_{3,2}$ on $3$ generators is the simply connected, connected nilpotent Lie group defined by the relations
\[[X_1,X_2] = Y_3, \quad [X_2,X_3] = Y_1, \quad [X_3,X_1] = Y_2,\]
where $X_1,X_2,X_3,Y_1,Y_2,Y_3$ is a basis of its Lie algebra (that is, the Lie algebra of the left-invariant vector fields on $N_{3,2}$). In exponential coordinates, $N_{3,2}$ can be identified with $\R^3_x \times \R^3_y$, where the group law is given by
\[(x,y) \cdot (x',y') = (x+x',y+y' + x\wedge x'/2)\]
and $x \wedge x'$ denotes the usual vector product of $x,x' \in \R^3$.
The family $(\delta_t)_{t > 0}$ of automorphic dilations of $N_{3,2}$, defined by
\[\delta_t(x,y) = (tx, t^2 y),\]
turns $N_{3,2}$ into a stratified group of homogeneous dimension $Q = 9$.

Let $L = -(X_1^2 + X_2^2 + X_3^2)$ be the homogeneous sublaplacian on $N_{3,2}$. $L$ is a self-adjoint operator on $L^2(N_{3,2})$, hence a functional calculus for $L$ is defined via spectral integration and, for all Borel functions $F : \R \to \C$, the operator $F(L)$ is bounded on $L^2(N_{3,2})$ whenever the ``spectral multiplier'' $F$ is a bounded function. Here we are interested in giving a sufficient condition for the $L^p$-boundedness (for $p \neq 2$) of the operator $F(L)$, in terms of smoothness properties of the multiplier $F$.

Let $W_2^s(\R)$ denote the $L^2$ Sobolev space of (fractional) order $s$. Then our main result reads as follows.

\begin{thm}\label{thm:mhn32}
Suppose that a function $F : \R \to \C$ satisfies
\[\sup_{t > 0} \|\eta \,F(t \cdot) \|_{W_2^s} < \infty\]
for some $s > 6/2$ and some nonzero $\eta \in C^\infty_c(\leftopenint 0,\infty \rightopenint)$. Then the operator $F(L)$ is of weak type $(1,1)$ and bounded on $L^p(N_{3,2})$ for all $p \in \leftopenint 1,\infty\rightopenint$.
\end{thm}

Observe that the general multiplier theorem for homogeneous sublaplacians on stratified Lie groups by Christ \cite{christ_multipliers_1991} and Mauceri and Meda \cite{mauceri_vectorvalued_1990} requires the stronger regularity condition $s > Q/2 = 9/2$. To the best of our knowledge, in the case of $N_{3,2}$ none of the results and techniques known so far allowed one to go below the condition $s > Q/2$. Our result pushes the regularity assumption down to $s > d/2 = 6/2$, where $d = 6$ is the topological dimension of $N_{3,2}$. We conjecture that this condition is sharp.

The problem of $L^p$-boundedness for spectral multipliers on nilpotent Lie groups has a long history, and the theorem by Christ and Mauceri and Meda is itself an improvement of a series of previous results (see, e.g., \cite{de_michele_heisenberg_1979,folland_hardy_1982,de_michele_mulipliers_1987}). Nevertheless it is still an open question, whether the homogeneous dimension in the smoothness condition may always be replaced by the topological dimension.

It has been known for a long time \cite{hebisch_multiplier_1993,mller_spectral_1994} that such an improvement of the multiplier theorem holds true in the case of the Heisenberg and related groups (more precisely, for direct products of M\'etivier and abelian groups; see also \cite{hebisch_multiplier_1995,martini_joint_2012}). This class of groups, however, does not include $N_{3,2}$, nor any free $2$-step nilpotent group $N_{n,2}$ on $n$ generators (see \cite[\S3]{rothschild_hypoelliptic_1976} for a definition), except for the smallest one, $N_{2,2}$, which is the $3$-dimensional Heisenberg group. The free groups $N_{n,2}$ have in a sense the maximal structural complexity among $2$-step groups, since every $2$-step nilpotent Lie group is a quotient of a free one. Our result should then hopefully shed some new light and contribute to the understanding of the problem for general $2$-step nilpotent Lie groups.

\section{Strategy of the proof}

The sublaplacian $L$ is a left-invariant operator on $N_{3,2}$, hence any operator of the form $F(L)$ is left-invariant too. Let $\Kern_{F(L)}$ then denote the convolution kernel of $F(L)$. As shown, e.g., in \cite[Theorem 4.6]{martini_joint_2012}, the previous Theorem~\ref{thm:mhn32} is a consequence of the following $L^1$-estimate.

\begin{prp}\label{prp:l1estimate}
For all $s > 6/2$, for all compact sets $K \subseteq \leftopenint 0,\infty \rightopenint$, and for all functions $F : \R \to \C$ such that $\supp F \subseteq K$,
\begin{equation}\label{eq:l1estimate}
\|\Kern_{F(L)}\|_1 \leq C_{K,s} \|F\|_{W_2^s}.
\end{equation}
\end{prp}

Let $|\cdot|_\delta$ be any $\delta_t$-homogeneous norm on $N_{3,2}$; take, e.g., $|(x,y)|_\delta = |x| + |y|^{1/2}$. The crucial estimate in the proof of \cite{mauceri_vectorvalued_1990} of the general theorem for stratified groups, that is,
\begin{equation}\label{eq:standardl2}
\|(1+|\cdot|_\delta)^{\alpha} \Kern_{F(L)}\|_2 \leq C_{K,\alpha,\beta} \|F\|_{W_2^\beta}
\end{equation}
for all $\alpha \geq 0$ and $\beta > \alpha$, implies \eqref{eq:l1estimate} when $s > 9/2$, by H\"older's inequality. In order to push the condition down to $s > 6/2$, here we prove an enhanced version of \eqref{eq:standardl2}, that is,
\begin{equation}\label{eq:weightedl2model}
\|(1+|\cdot|_\delta)^{\alpha} \, w^r \, \Kern_{F(L)}\|_2 \leq C_{K,\alpha,\beta,r} \|F\|_{W_2^\beta},
\end{equation}
for some ``extra weight'' function $w$ on $N_{3,2}$, and suitable constraints on the exponents $\alpha,\beta,r$.

A similar approach is adopted in the mentioned works on the Heisenberg and related groups. However, in \cite{mller_spectral_1994} the extra weight $w$ is the full weight $1+|\cdot|_\delta$, while \cite{hebisch_multiplier_1993} employs the weight $w(x,y) = 1+|x|$. Here instead the weight $w(x,y) = 1+|y|$ is used, and \eqref{eq:weightedl2model} is proved under the conditions $\alpha \geq 0$, $0 \leq r < 3/2$, $\beta > \alpha+r$ (see Proposition~\ref{prp:improvedl2estimate} below).

The proof of \eqref{eq:weightedl2model} when $\alpha = 0$ is based on a careful analysis exploiting identities for Laguerre polynomials, somehow in the spirit of \cite{de_michele_heisenberg_1979,mller_spectral_1994,mller_marcinkiewicz_1996}, but with additional complexity due, inter alia, to the simultanous use of generalized Laguerre polynomials of different types. The estimate for arbitrary $\alpha$ is then recovered by interpolation with \eqref{eq:standardl2}. An analogous strategy is followed in \cite{martini_grushin2}, where identities for Hermite polynomials are used in order to prove a sharp spectral multiplier theorem for Grushin operators.

\section{A joint functional calculus}

It is convenient for us to embed the functional calculus for the sublaplacian $L$ in a larger functional calculus for a system of commuting left-invariant differential operators on $N_{3,2}$. Specifically, the operators
\begin{equation}\label{eq:operators}
L,-iY_1,-iY_2,-iY_3
\end{equation}
are essentially self-adjoint and commute strongly, hence they admit a joint functional calculus (see, e.g., \cite{martini_spectral_2011}).

If $\vecY$ denotes the ``vector of operators'' $(-iY_1,-iY_2,-iY_3)$, then we can express the convolution kernel $\Kern_{G(L,\vecY)}$ of the operator $G(L,\vecY)$ in terms of Laguerre functions (cf.\ \cite{fischer_gelfand_2009}). Namely, for all $n,k \in \N$, let
\[L_n^{(k)}(u) = \frac{u^{-k} e^u}{n!} \left( \frac{d}{du} \right)^n (u^{k+n} e^{-u})\]
be the $n$-th Laguerre polynomial of type $k$, and define
\[\Ell_n^{(k)}(t) = 2(-1)^n e^{-t} L_n^{(k)}(2t).\]
Further, for all $\eta \in \R^3 \setminus \{0\}$ and $\xi \in \R^3$, define $\xi^\eta_\parallel$ and $\xi^\eta_\perp$ by
\[\xi^\eta_\parallel = \langle \xi, \eta/|\eta| \rangle, \qquad \xi^\eta_\perp = \xi - \xi^\eta_\parallel \eta/|\eta|.\]

\begin{prp}\label{prp:kernel}
Let $G : \R^4 \to \C$ be in the Schwartz class, and set 
\begin{equation}\label{eq:reparametrization}
m(n,\mu,\eta) = G((2n+1)|\eta| + \mu^2, \eta),
\end{equation}
for all $n \in \N$, $\mu \in \R$, $\xi,\eta \in \R^3$ with $\eta \neq 0$. Then
\[
\Kern_{G(L,\vecY)}(x,y) = \frac{1}{(2\pi)^{6}} \int_{\R^3} \int_{\R^3} \sum_{n \in \N} m(n,\xi^\eta_\parallel,\eta)
\, \Ell_n^{(0)}(|\xi^\eta_\perp|^2 /|\eta|) 
\, e^{i \langle \xi, x \rangle} \, e^{i \langle \eta, y \rangle} \,d\xi \,d\eta.
\]
\end{prp}
\begin{proof}
For all $\eta \in \R^3 \setminus \{0\}$, choose a unit vector $E_\eta \in \eta^\perp$, and set $\bar E_\eta = (\eta/|\eta|) \wedge E_\eta$; moreover, for all $x \in \R^3$, denote by $x_1^\eta$, $x_2^\eta$, $x_\parallel^\eta$ the components of $x$ with respect to the positive orthonormal basis $E^\eta$, $\bar E^\eta$, $\eta/|\eta|$ of $\R^3$.

For all $\eta \in \R^3 \setminus \{0\}$ and all $\mu \in \R$, an irreducible unitary representation $\pi_{\eta,\mu}$ of $N_{3,2}$ on $L^2(\R)$ is defined by
\[
\pi_{\eta,\mu}(x,y) \phi(u) = e^{i \langle \eta , y \rangle} e^{i |\eta| (u + x_1^\eta/2) x_2^\eta} e^{i \mu x_\parallel^\eta} \phi(x_1^\eta + u)
\]
for all $(x,y) \in N_{3,2}$, $u \in \R$, $\phi \in L^2(\R)$.
Following, e.g., \cite[\S2]{astengo_hardys_2000}, one can see that these representations are sufficient to write the Plancherel formula for the group Fourier transform of $N_{3,2}$, and the corresponding Fourier inversion formula:
\begin{equation}\label{eq:groupinversion}
f(x,y) = (2\pi)^{-5} \int_{\R^3 \setminus \{0\}} \int_{\R} \tr (\pi_{\eta,\mu}(x,y) \, \pi_{\eta,\mu}(f)) \, |\eta| \, d\mu \, d\eta
\end{equation}
for all $f : N_{3,2} \to \C$ in the Schwartz class and all $(x,y) \in N_{3,2}$, where $\pi_{\eta,\mu}(f) = \int_{N_{3,2}} f(z) \, \pi_{\eta,\mu}(z^{-1}) \,dz$.

Fix $\eta \in \R^3 \setminus \{0\}$ and $\mu \in \R$. The operators \eqref{eq:operators} are represented in $\pi_{\eta,\mu}$ as
\begin{equation}\label{eq:repops}
d\pi_{\eta,\mu}(L) = - \partial_u^2 + |\eta|^2 u^2 + \mu^2, \qquad d\pi_{\eta,\mu}(-iY_j) = \eta_j.
\end{equation}
If $h_n$ is the $n$-th Hermite function, that is,
\[
h_n(t) = (-1)^n (n! \, 2^n \sqrt{\pi})^{-1/2} e^{t^2/2} \left(\frac{d}{dt}\right)^n e^{-t^2},
\]
and $\tilde h_{\eta,n}$ is defined by
\[
\tilde h_{\eta,n}(u) = |\eta|^{1/4} h_n(|\eta|^{1/2} u),
\]
then $\{\tilde h_{\eta,n}\}_{n \in \N}$ is a complete orthonormal system for $L^2(\R)$, made of joint eigenfunctions of the operators \eqref{eq:repops}; in fact,
\begin{equation}\label{eq:eigenvalues}
\begin{aligned}
d\pi_{\eta,\mu}(L) \tilde h_{\eta,n} &= (|\eta|(2n+1) + \mu^2) \tilde h_{\eta,n}, \\
d\pi_{\eta,\mu}(-iY_j) \tilde h_{\eta,n} &= \eta_j \tilde h_{\eta,n}.
\end{aligned}
\end{equation}
Moreover the corresponding diagonal matrix coefficients $\varphi_{\eta,\mu,n}$ of $\pi_{\eta,\mu}$ are given by
\[\begin{split}
\varphi_{\eta,\mu,n}(&x,y) = \langle \pi_{\eta,\mu}(x,y) \tilde h_{\eta,n}, \tilde h_{\eta,n} \rangle \\
&= e^{i \langle \eta , y \rangle} e^{i \mu x_\parallel^\eta} |\eta|^{1/2} \int_\R e^{i |\eta| u x_2^\eta} \, h_{n}(|\eta|^{1/2} (u+x_1^\eta/2)) \, h_{n}(|\eta|^{1/2} (u-x_1^\eta/2)) \,du.
\end{split}\]
The last integral is essentially the Fourier-Wigner transform of the pair $(h_n,h_n)$, whose Fourier transform has a particularly simple expression (cf.\ \cite[formula (1.90)]{folland_harmonic_1989}); the parity of the Hermite functions then yields
\[\begin{split}
\varphi_{\eta,\mu,n}(&x,y) = 
e^{i \langle \eta , y \rangle} e^{i \mu x_\parallel^\eta} \frac{(-1)^n}{\pi |\eta|} \int_{\R^2} e^{i v_2 x_2^\eta} e^{i v_1 x_1^\eta} \\
&\qquad \times \int_\R e^{-it (2v_1/ |\eta|^{1/2}) } \, h_{n}(t+v_2/|\eta|^{1/2}) \, h_{n}(t-v_2/|\eta|^{1/2}) \,dt \,dv ,\\
\end{split}\]
that is,
\begin{equation}\label{eq:diagonalcoefficients}
\varphi_{\eta,\mu,n}(x,y) = \frac{1}{2\pi |\eta|} e^{i \langle \eta , y \rangle} e^{i \mu x_\parallel^\eta} \int_{\R^2} e^{i v_1 x_1^\eta} e^{i v_2 x_2^\eta} \Ell_n^{(0)}(|v|^2/|\eta|) \,dv
\end{equation}
(see \cite[Theorem 1.3.4]{thangavelu_lectures_1993} or \cite[Theorem 1.104]{folland_harmonic_1989}).

Note that $\Kern_{G(L,\vecY)} \in \Sz(N_{3,2})$ since $G \in \Sz(\R^4)$ (see \cite[Theorem 5.2]{astengo_gelfand_2009} or \cite[\S4.2]{martini_multipliers_2010}). Moreover 
\[\pi_{\eta,\mu}(\Kern_{G(L,\vecY)}) \tilde h_{\eta,n} = G(|\eta|(2n+1) + \mu^2,\eta) \tilde h_{\eta,n}\]
by \eqref{eq:eigenvalues} and \cite[Proposition 1.1]{mller_restriction_1990}, hence
\[\langle \pi_{\eta,\mu}(x,y) \pi_{\eta,\mu}(\Kern_{G(L,\vecY)}) \tilde h_{\eta,n}, \tilde h_{\eta,n} \rangle = m(n,\mu,\eta) \, \varphi_{\eta,\mu,n}(x,y).\]
Therefore, by \eqref{eq:groupinversion} and \eqref{eq:diagonalcoefficients},
\[\begin{split}
&\Kern_{G(L,\vecY)}(x,y) \\
&\quad= (2\pi)^{-5} \int_{\R^3 \setminus \{0\}} \int_{\R} \sum_{n \in \N} m(n,\mu,\eta) \, \varphi_{\eta,\mu,n}(x,y) \, |\eta| \, d\mu \, d\eta \\
&\quad= (2\pi)^{-6} \int_{\R^3} \int_{\R^3} \sum_{n \in \N} m(n,\xi_3,\eta) \, e^{i \langle \eta , y \rangle} 
e^{i \langle \xi, (x_1^\eta,x_2^\eta,x_\parallel^\eta) \rangle}  \Ell_n^{(0)}((\xi_1^2+\xi_2^2)/|\eta|) \,d\xi \,d\eta.
\end{split}\]
The conclusion follows by a change of variable in the inner integral.
\end{proof}

\section{Weighted estimates}

For convenience, set $\Ell_n^{(k)} = 0$ for all  $n < 0$. The following identities are easily obtained from the properties of Laguerre polynomials (see, e.g., \cite[\S10.12]{erdelyi_higher2_1981}).

\begin{lem}
For all $k,n,n' \in \N$ and $t \in \R$,
\begin{gather}
\label{eq:laguerrepm} \Ell_n^{(k)}(t) = \Ell_{n-1}^{(k+1)}(t) + \Ell_n^{(k+1)}(t), \\
\label{eq:laguerred} \frac{d}{dt} \Ell_n^{(k)}(t) = \Ell_{n-1}^{(k+1)}(t) - \Ell_{n}^{(k+1)}(t),\\
\label{eq:laguerreo} \int_{0}^{\infty} \Ell_n^{(k)}(t) \, \Ell_{n'}^{(k)}(t) \, t^k \,dt = \begin{cases}
\frac{(n+k)!}{2^{k-1} n!} &\text{if $n=n'$,}\\
0 &\text{otherwise.}
\end{cases}
\end{gather}
\end{lem}

We introduce some operators on functions $f : \N \times \R \times \R^3 \to \C$:
\begin{align*}
\tau f(n,\mu,\eta) &= f(n+1,\mu,\eta), \\
\delta f(n,\mu,\eta) &= f(n+1,\mu,\eta) - f(n,\mu,\eta), \\
\partial_\mu f(n,\mu,\eta) &= \frac{\partial}{\partial \mu} f(n,\mu,\eta), \\
\partial_\eta^\alpha f(n,\mu,\eta) &= \left(\frac{\partial}{\partial \eta}\right)^\alpha f(n,\mu,\eta),
\end{align*}
for all $\alpha \in \N^3$. For all multiindices $\alpha \in \N^3$, we denote by $|\alpha|$ its length $\alpha_1+\alpha_2+\alpha_3$. We set moreover $\langle t \rangle = 2|t|+1$ for all $t \in \R$.

Note that, for all compactly supported $f : \N \times \R \times \R^3 \to \C$, $\tau^l f$ is null for all sufficiently large $l \in \N$; hence the operator $1+\tau$, when restricted to the set of compactly supported functions, is invertible, with inverse given by
\[
(1+\tau)^{-1} f = \sum_{l \in \N} (-1)^l \tau^l f,
\]
and therefore the operator $(1+\tau)^q$ is well-defined for all $q \in \Z$.

\begin{prp}
Let $G : \R^4 \to \C$ be smooth and compactly supported in $\R \times (\R^3 \setminus \{0\})$, and let $m(n,\mu,\eta)$ be defined by \eqref{eq:reparametrization}.
For all $\alpha \in \N^3$,
\begin{multline}\label{eq:weightedl2}
\int_{N_{3,2}} |y^\alpha \, \Kern_{G(L,\vecY)}(x,y)|^2 \,dx\,dy \\
\leq C_\alpha \sum_{\iota \in I_\alpha} \sum_{n \in \N} \int_{\R^3} \int_\R |\partial_\eta^{\gamma^\iota} \partial_\mu^{l_\iota} \delta^{k_\iota} (1+\tau)^{|\beta^\iota|-k_\iota} m(n,\mu,\eta)|^2 \\
\times  \mu^{2b_\iota} \, |\eta|^{2|\gamma^\iota|-2|\alpha|-2k_\iota+|\beta^\iota|+1} \langle n \rangle^{|\beta^\iota|} \,d\mu \,d\eta,
\end{multline}
where $I_\alpha$ is a finite set and, for all $\iota \in I_\alpha$,
\begin{itemize}
\item $\gamma^\iota \in \N^3$, $l_\iota, k_\iota \in \N$, $\gamma^\iota \leq \alpha$, $\min\{1,|\alpha|\} \leq |\gamma^\iota| + l_\iota + k_\iota \leq |\alpha|$,
\item $b_\iota \in \N$, $\beta^\iota \in \N^3$, $b_\iota + |\beta^\iota| = l_\iota + 2k_\iota$, $|\gamma^\iota| + l_\iota + b_\iota \leq |\alpha|$.
\end{itemize}
\end{prp}
\begin{proof}
Proposition~\ref{prp:kernel} and integration by parts allow us to write
\begin{multline}\label{eq:weightedkernel}
y^\alpha \Kern_{G(L,\vecY)}(x,y) \\
 = \frac{i^{|\alpha|}}{(2\pi)^{6}} \int_{\R^3} \int_{\R^3} \left[ \left(\frac{\partial}{\partial \eta}\right)^\alpha \sum_{n \in \N} m(n,\xi^\eta_\parallel,\eta)
\, \Ell_n^{(0)}(|\xi^\eta_\perp|^2 /|\eta|) \right]
\, e^{i \langle \xi, x \rangle} \, e^{i \langle \eta, y \rangle} \,d\xi \,d\eta.
\end{multline}
From the definition of $\xi^\eta_\parallel$ and $\xi^\eta_\perp$, the following identities are not difficult to obtain:
\begin{equation}
\begin{gathered}
\frac{\partial}{\partial \eta_j} \xi_\parallel^\eta = (\xi_\perp^\eta)_j \frac{1}{|\eta|}, 
\qquad
\frac{\partial}{\partial \eta_j} (\xi_\perp^\eta)_k = - \xi_\parallel^\eta \frac{\partial}{\partial \eta_j} \frac{\eta_k}{|\eta|} - (\xi_\perp^\eta)_j \frac{\eta_k}{|\eta|^2}, \\
\frac{\partial}{\partial \eta_j} \frac{|\xi_\perp^\eta|^2}{|\eta|} = - \xi_\parallel^\eta (\xi_\perp^\eta)_j \frac{2}{|\eta|^2} - |\xi_\perp^\eta|^2 \frac{\eta_j}{|\eta|^3}.
\end{gathered}
\end{equation}
The multiindex notation will also be used as follows:
\[(\xi^\eta_\perp)^\beta = (\xi^\eta_\perp)_1^{\beta_1} (\xi^\eta_\perp)_2^{\beta_2} (\xi^\eta_\perp)_3^{\beta_3}\]
for all $\xi,\eta \in \R$, with $\eta \neq 0$, and all $\beta \in \N^3$; consequently
\[|\xi^\eta_\perp|^2 = (\xi^\eta_\perp)^{(2,0,0)} + (\xi^\eta_\perp)^{(0,2,0)} + (\xi^\eta_\perp)^{(0,0,2)}.\]
Via these identities, one can prove inductively that, for all $\alpha \in \N^3$,
\begin{multline}\label{eq:kernelderivative}
\left(\frac{\partial}{\partial \eta}\right)^\alpha \sum_{n \in \N} m(n,\xi^\eta_\parallel,\eta) \, \Ell_n^{(0)}(|\xi^\eta_\perp|^2 /|\eta|) \\
= \sum_{\iota \in I_\alpha} \sum_{n \in \N} \partial_\eta^{\gamma^\iota} \partial_\mu^{l_\iota} \delta^{k_\iota} m(n,\xi_\parallel^\eta,\eta) \, (\xi_\parallel^\eta)^{b_\iota} \, (\xi_\perp^\eta)^{\beta^\iota} \, \Theta_\iota(\eta) \, \Ell_n^{(k_\iota)}(|\xi_\perp^\eta|^2/|\eta|),
\end{multline}
where $I_\alpha$, $\gamma^\iota$, $l_\iota$, $k_\iota$, $b_\iota$, $\beta^\iota$ are as in the statement above, while $\Theta_\iota : \R^3 \setminus \{0\} \to \R$ is smooth and homogeneous of degree $|\gamma^\iota| - |\alpha| - k_\iota$. For the inductive step, one employs Leibniz' rule, and when a derivative hits a Laguerre function, the identity \eqref{eq:laguerred} together with summation by parts is used.

Note that, for all compactly supported $f : \N \times \R \times \R^3 \to \C$,
\[
\sum_{n \in \N} f(n,\mu,\eta) \, \Ell_n^{(k)}(t) = \sum_{n \in \N} (1+\tau) f(n,\mu,\eta) \, \Ell_n^{(k+1)}(t),
\]
by \eqref{eq:laguerrepm}. Since $1+\tau$ is invertible, simple manipulations and iteration yield the more general identity
\[
\sum_{n \in \N} f(n,\mu,\eta) \, \Ell_n^{(k)}(t) = \sum_{n \in \N} (1+\tau)^{k'-k} f(n,\mu,\eta) \, \Ell_n^{(k')}(t),
\]
for all $k,k' \in \N$. This formula allows us to adjust in \eqref{eq:kernelderivative} the type of the Laguerre functions to the exponent of $\xi_\perp$, and to obtain that
\begin{multline*}
\left(\frac{\partial}{\partial \eta}\right)^\alpha \sum_{n \in \N} m(n,\xi^\eta_\parallel,\eta) \, \Ell_n^{(0)}(|\xi^\eta_\perp|^2 /|\eta|) \\
= \sum_{\iota \in I_\alpha} \sum_{n \in \N} \partial_\eta^{\gamma^\iota} \partial_\mu^{l_\iota} \delta^{k_\iota} (1+\tau)^{|\beta^\iota|-k_\iota} m(n,\xi_\parallel^\eta,\eta) \, (\xi_\parallel^\eta)^{b_\iota} \, (\xi_\perp^\eta)^{\beta^\iota} \, \Theta_\iota(\eta) \, \Ell_n^{(|\beta^\iota|)}(|\xi_\perp^\eta|^2/|\eta|),
\end{multline*}

By plugging this identity into \eqref{eq:weightedkernel} and exploiting Plancherel's formula for the Fourier transform, the finiteness of $I_\alpha$ and the triangular inequality, we get that
\begin{multline*}
\int_{N_{3,2}} |y^\alpha \Kern_{G(L,\vecY)}(x,y)|^2 \,dx \,dy \\
\leq C_\alpha \sum_{\iota \in I_\alpha} \int_{\R^3} \int_{\R} \int_{\R^2} \left|\sum_{n \in \N} \partial_\eta^{\gamma^\iota} \partial_\mu^{l_\iota} \delta^{k_\iota} (1+\tau)^{|\beta^\iota|-k_\iota} m(n,\mu,\eta)  \, \Ell_n^{(|\beta^\iota|)}(|\zeta|^2/|\eta|) \right|^2 \\
\times \mu^{2b_\iota} \, |\zeta|^{2|\beta^\iota|} \, |\eta|^{2|\gamma^\iota| - 2|\alpha| - 2k_\iota}
 \,d\zeta \,d\mu \,d\eta 
\end{multline*}
A passage to polar coordinates in the $\zeta$-integral and a rescaling then give that
\begin{multline*}
\int_{N_{3,2}} |y^\alpha \Kern_{G(L,\vecY)}(x,y)|^2 \,dx \,dy \\
\leq C_\alpha \sum_{\iota \in I_\alpha} \int_{\R^3} \int_{\R} \int_0^\infty \left|\sum_{n \in \N} \partial_\eta^{\gamma^\iota} \partial_\mu^{l_\iota} \delta^{k_\iota} (1+\tau)^{|\beta^\iota|-k_\iota} m(n,\mu,\eta)  \, \Ell_n^{(|\beta^\iota|)}(s) \right|^2 s^{|\beta^\iota|} \,ds \\
\times \mu^{2b_\iota} \, |\eta|^{2|\gamma^\iota| - 2|\alpha| - 2k_\iota+|\beta^\iota|+1} \,d\mu \,d\eta,
\end{multline*}
and the conclusion follows by applying the orthogonality relations \eqref{eq:laguerreo} for the Laguerre functions to the inner integral.
\end{proof}

Note that $\tau f(\cdot,\mu,\eta)$, $\delta f(\cdot,\mu,\eta)$ depend only on $f(\cdot,\mu,\eta)$; in other words, $\tau$ and $\delta$ can be considered as operators on functions $\N \to \C$. The next lemma will be useful in converting finite differences into continuous derivatives.

\begin{lem}\label{lem:discretecontinuous}
Let $f : \N \to \C$ have a smooth extension $\tilde f : \leftclosedint 0,\infty \rightopenint \to \C$, and let $k \in \N$. Then
\[\delta^k f(n) = \int_{J_k} \tilde f^{(k)}(n+s) \,d\nu_k(s)\]
for all $n \in \N$, where $J_k = \leftclosedint 0,k\rightclosedint$ and $\nu_k$ is a Borel probability measure on $J_k$. In particular
\[|\delta^k f(n)|^2 \leq \int_{J_k} |\tilde f^{(k)}(n+s)|^2 \,d\nu_k(s)\]
for all $n \in \N$.
\end{lem}
\begin{proof}
Iterated application of the fundamental theorem of integral calculus gives
\[\delta^k f(n) = \int_{\leftclosedint 0,1\rightclosedint^k} \tilde f^{(k)}(n+s_1+\dots+s_k) \,ds.\]
The conclusion follows by taking as $\nu_k$ the push-forward of the uniform distribution on $\leftclosedint 0,1\rightclosedint^k$ via the map $(s_1,\dots,s_k) \mapsto s_1+\dots+s_k$, and by H\"older's inequality.
\end{proof}

We give now a simplified version of the right-hand side of \eqref{eq:weightedl2}, in the case where we restrict to the functional calculus for the sublaplacian $L$ alone. In order to avoid divergent series, however, it is convenient at first to truncate the multiplier along the spectrum of $\vecY$.

\begin{lem}\label{lem:weightedplancherel}
Let $\chi \in C^\infty_c(\R)$ be supported in $\leftclosedint 1/2,2 \rightclosedint$, $K \subseteq \leftopenint 0, \infty \rightopenint$ be compact and $M \in \leftopenint 0, \infty \rightopenint$. If $F : \R \to \C$ is smooth and supported in $K$, and $F_M : \R \times \R^3 \to \C$ is given by
\[F_M(\lambda,\eta) = F(\lambda) \, \chi(|\eta|/M),\]
then, for all $r \in \leftclosedint 0, \infty \rightopenint$,
\[\int_{N_{3,2}} | |y|^r \, \Kern_{F_M(L,\vecY)}(x,y) |^2 \,dx \,dy \leq C_{K,\chi,r} \, M^{3-2r} \|F\|_{W_2^r}^2.\]
\end{lem}
\begin{proof}
We may restrict to the case $r \in \N$, the other cases being recovered a posteriori by interpolation. Hence we need to prove that
\begin{equation}\label{eq:weightedmultiindex}
\int_{N_{3,2}} | y^\alpha \, \Kern_{F_M(L,\vecY)}(x,y) |^2 \,dx \,dy \leq C_{K,\chi,\alpha} \, M^{3-2|\alpha|} \|F\|_{W_2^{|\alpha|}}^2
\end{equation}
for all $\alpha \in \N^3$. On the other hand, if
\[m(n,\mu,\eta) = F( |\eta| \langle n \rangle + \mu^2 ) \, \chi(|\eta|/M),\]
then the left-hand side of \eqref{eq:weightedmultiindex} can be majorized by \eqref{eq:weightedl2}, and we are reduced to proving
\begin{multline}\label{eq:weightedsummand}
\sum_{n\in \N}\int_{\R^3} \int_\R |\partial_\eta^{\gamma^\iota} \partial_\mu^{l_\iota} \delta^{k_\iota} (1+\tau)^{|\beta^\iota|-k_\iota} m(n,\mu,\eta)|^2 
\, \mu^{2b_\iota} \, |\eta|^{2|\gamma^\iota|-2|\alpha|-2k_\iota+|\beta^\iota|+1} \\
\times \langle n \rangle^{|\beta^\iota|} \,d\mu \,d\eta 
\leq C_{K,\chi,\alpha} \, M^{3-2|\alpha|} \|F\|_{W_2^{|\alpha|}}^2
\end{multline}
for all $\iota \in I_\alpha$.

Consider first the case $|\beta^\iota| \geq k_\iota$. A smooth extension $\tilde m : \R \times \R \times \R^3 \to \C$ of $m$ is defined by
\[\tilde m(t,\mu,\eta) = F( |\eta| (2t+1) + \mu^2 ) \, \chi(|\eta|/M).\]
Then, by Lemma~\ref{lem:discretecontinuous},
\begin{multline*}
\partial_\eta^{\gamma^\iota} \partial_\mu^{l_\iota} \delta^{k_\iota} (1+\tau)^{|\beta^\iota|-k_\iota} m(n,\mu,\eta) \\
= \sum_{j=0}^{|\beta^\iota| - k_\iota} \binom{|\beta^\iota|-k_\iota}{j} \int_{J_\iota} \partial_{\eta}^{\gamma^\iota} \partial_\mu^{l_\iota} \partial_t^{k_\iota} \tilde m(n+j+s,\mu,\eta) \,d\nu_\iota(s),
\end{multline*}
where $J_\iota = \leftclosedint 0,k_\iota \rightclosedint$ and $\nu_\iota$ is a suitable probability measure on $J_\iota$; consequently \eqref{eq:weightedsummand} will be proved if we show that
\begin{multline}\label{eq:weightedsummand2}
\sum_{n \in \N} \int_{\R^3} \int_\R |\partial_\eta^{\gamma^\iota} \partial_\mu^{l_\iota} \partial_t^{k_\iota} \tilde m(n+s,\mu,\eta)|^2 
\, \mu^{2b_\iota} \, |\eta|^{2|\gamma^\iota|-2|\alpha|-2k_\iota+|\beta^\iota|+1} \\
\times \langle n \rangle^{|\beta^\iota|} \,d\mu \,d\eta 
\leq C_{K,\chi,\alpha} \, M^{3-2|\alpha|} \|F\|_{W_2^{|\alpha|}}^2
\end{multline}
for all $s \in \leftclosedint 0,|\beta^\iota|\rightclosedint$. On the other hand, it is easily proved inductively that
\begin{multline*}
\partial_\eta^{\gamma^\iota} \partial_\mu^{l_\iota} \partial_t^{k_\iota} \tilde m(t,\mu,\eta) \\
= \sum_{r=\lceil l_\iota/2\rceil}^{l_\iota} \sum_{v=0}^{|\gamma^\iota|} \sum_{q=0}^{|\gamma^\iota| - v} \Psi_{i,v,q}(\eta) \, \langle t \rangle^{v} \mu^{2r-l_\iota} M^{-q} F^{(k_\iota+v+r)}(|\eta|\langle t \rangle + \mu^2) \, \chi^{(q)}(|\eta|/M)
\end{multline*}
for all $t \geq 0$, where $\Psi_{i,q,v} : \R^3 \setminus \{0\} \to \R$ is smooth and homogeneous of degree $k_\iota + v + q - |\gamma^\iota|$; hence
\begin{multline}\label{eq:multiplierderivatives}
|\partial_\eta^{\gamma^\iota} \partial_\mu^{l_\iota} \partial_t^{k_\iota} \tilde m(t,\mu,\eta)|^2 
\leq C_{\chi,\alpha} \sum_{r=\lceil l_\iota/2\rceil}^{l_\iota} \sum_{v=0}^{|\gamma^\iota|} M^{2k_\iota+2v-2|\gamma^\iota|} 
\langle t \rangle^{2v} \mu^{4r-2l_\iota} \\
\times |F^{(k_\iota+v+r)}(|\eta|\langle t \rangle + \mu^2)|^2 \, \tilde\chi(|\eta|/M),
\end{multline}
where $\tilde\chi$ is the characteristic function of $\leftclosedint 1/2,2\rightclosedint$, and we are using the fact that $|\eta| \sim M$ in the region where $\tilde\chi(|\eta|/M) \neq 0$. Consequently the left-hand side of \eqref{eq:weightedsummand2} is majorized by
\[\begin{split}
C_{\chi,\alpha} &\sum_{r=\lceil l_\iota/2\rceil}^{l_\iota} \sum_{v=0}^{|\gamma^\iota|} M^{2v-2|\alpha| + |\beta^\iota| + 1} \sum_{n\in \N} \langle n \rangle^{|\beta^\iota|} \langle n+s \rangle^{2v} \\
&\qquad\times \int_{\R^3} \int_{\R} |F^{(k_\iota+v+r)}(|\eta|\langle n+s \rangle + \mu^2)|^2 \, \mu^{2b_\iota + 4r-2l_\iota} \, \tilde\chi(|\eta|/M) \,d\mu \,d\eta \\
&\leq C_{\chi,\alpha} \sum_{r=\lceil l_\iota/2\rceil}^{l_\iota} \sum_{v=0}^{|\gamma^\iota|} M^{2v-2|\alpha| + |\beta^\iota| + 3} \sum_{n\in \N} \langle n+s \rangle^{|\beta^\iota|+2v} \\
&\qquad\times \int_0^\infty \int_0^\infty |F^{(k_\iota+v+r)}(\rho\langle n+s \rangle + \mu^2)|^2 \, \mu^{2b_\iota + 4r-2l_\iota} \, \tilde\chi(\rho/M) \,d\mu \, \,d\rho \\
&\leq C_{\chi,\alpha} \sum_{r=\lceil l_\iota/2\rceil}^{l_\iota} \sum_{v=0}^{|\gamma^\iota|} M^{2v-2|\alpha| + |\beta^\iota| + 3}
\int_{0}^\infty \int_{0}^\infty |F^{(k_\iota+v+r)}(\rho + \mu^2)|^2 \\
&\qquad\times \mu^{2b_\iota + 4r- 2l_\iota} \, 
\sum_{n\in \N} \langle n+s \rangle^{|\beta^\iota| + 2v - 1} \tilde\chi(\rho/(\langle n+s \rangle M)) \,d\mu \,d\rho,
\end{split}\]
by passing to polar coordinates and rescaling. The last sum in $n$ is easily controlled by $(\rho/M)^{|\beta^\iota|+2v}$, hence the left-hand side of \eqref{eq:weightedsummand2} is majorized by
\begin{multline*}
C_{\chi,\alpha} \, M^{3-2|\alpha|} \sum_{r=\lceil l_\iota/2\rceil}^{l_\iota} \sum_{v=0}^{|\gamma^\iota|} 
\int_{0}^\infty \int_{0}^\infty |F^{(k_\iota+v+r)}(\rho + \mu^2)|^2 \mu^{2b_\iota + 4r-2l_\iota} \, \rho^{|\beta^\iota| + 2v} \,d\mu \,d\rho \\
\leq C_{K,\chi,\alpha} \, M^{3-2|\alpha|} \sum_{r=\lceil l_\iota/2\rceil}^{l_\iota} \sum_{v=0}^{|\gamma^\iota|} 
\sup_{u \in \leftclosedint 0,\max K\rightclosedint} \int_{0}^\infty |F^{(k_\iota+v+r)}(\rho + u)|^2 \,d\rho,
\end{multline*}
because $2b_\iota + 4r-2l_\iota \geq 0$ and $|\beta^\iota| + 2v \geq 0$ if $r$ and $v$ are in the range of summation, and $\supp F \subseteq K$. Since moreover $k_\iota+v+r \leq k_\iota+\gamma^\iota+l_\iota \leq |\alpha|$, the last integral is dominated by $\|F\|_{W_2^{|\alpha|}}$ uniformly in $r,v,u$, and \eqref{eq:weightedsummand2} follows.

Consider now the case $|\beta^\iota| < k_\iota$. Via the identity
\[(1+\tau)^{-1} = (1-\tau) (1-\tau^2)^{-1} = - \delta (1-\tau^2)^{-1} = -\delta \sum_{j=0}^\infty \tau^{2j},\]
together with Lemma~\ref{lem:discretecontinuous}, we obtain that
\begin{multline}\label{eq:quasifinitesum}
\partial_\eta^{\gamma^\iota} \partial_\mu^{l_\iota} \delta^{k_\iota} (1+\tau)^{|\beta^\iota|-k_\iota} m(n,\mu,\eta) \\
= (-1)^{k_\iota-|\beta^\iota|} \sum_{j=0}^{\infty} {\textstyle \binom{j+k_\iota-|\beta^\iota|-1}{k_\iota-|\beta^\iota|-1}} \int_{J_\iota} \partial_{\eta}^{\gamma^\iota} \partial_\mu^{l_\iota} \partial_t^{2k_\iota-|\beta^\iota|} \tilde m(n+2j+s,\mu,\eta) \,d\nu_\iota(s),
\end{multline}
where $J_\iota = \leftclosedint 0,2k_\iota -|\beta^\iota|\rightclosedint$ and $\nu_\iota$ is a suitable probability measure on $J_\iota$. Note that, because of the assumptions on the supports of $F$ and $\chi$, the sum on $j$ in the right-hand side of \eqref{eq:quasifinitesum} is a finite sum, that is, the $j$-th summand is nonzero only if $\langle n +2j \rangle \leq 2 M^{-1} \max K$; consequently, by applying the Cauchy-Schwarz inequality to the sum in $j$, and by \eqref{eq:multiplierderivatives},
\begin{multline*}
|\partial_\eta^{\gamma^\iota} \partial_\mu^{l_\iota} \delta^{k_\iota} (1+\tau)^{|\beta^\iota|-k_\iota} m(n,\mu,\eta)|^2 \\
\leq C_{K,\alpha} \, M^{1+2|\beta^\iota|-2k_\iota} \sum_{j=0}^\infty \int_{J_\iota} |\partial_{\eta}^{\gamma^\iota} \partial_\mu^{l_\iota} \partial_t^{2k_\iota-|\beta^\iota|} \tilde m(n+2j+s,\mu,\eta)|^2 \,d\nu_\iota(s) \\
\leq C_{K,\chi,\alpha} \sum_{r=\lceil l_\iota/2\rceil}^{l_\iota} \sum_{v=0}^{|\gamma^\iota|} M^{1+2k_\iota+2v-2|\gamma^\iota|} 
\sum_{j=0}^\infty \int_{J_\iota} \langle n+2j+s \rangle^{2v} \mu^{4r-2l_\iota} \\
\times |F^{(2k_\iota-|\beta^\iota|+v+r)}(|\eta|\langle n+2j+s \rangle + \mu^2)|^2 \, \tilde\chi(|\eta|/M) \,d\nu_\iota(s).
\end{multline*}
Remember that $|\eta| \sim M$ in the region where $\tilde\chi(|\eta|/M) \neq 0$. Hence the left-hand side of \eqref{eq:weightedsummand} is majorized by
\[\begin{split}
&C_{K,\chi,\alpha} \sum_{r=\lceil l_\iota/2\rceil}^{l_\iota} \sum_{v=0}^{|\gamma^\iota|} \int_{J_\iota} \sum_{n\in\N} \sum_{j \in \N} \langle n+2j+s \rangle^{2v} \langle n \rangle^{|\beta^\iota|} \int_{\R^3} \int_{\R}  M^{2+2v-2|\alpha|+|\beta^\iota|} \\
&\quad\times \mu^{2b_\iota + 4r-2l_\iota} \, |F^{(2k_\iota-|\beta^\iota|+v+r)}(|\eta|\langle n+2j+s \rangle + \mu^2)|^2 \, \tilde\chi(|\eta|/M) \,d\mu \,d\eta \,d\nu_\iota(s) \\
&\leq C_{K,\chi,\alpha} \sum_{r=\lceil l_\iota/2\rceil}^{l_\iota} \sum_{v=0}^{|\gamma^\iota|} \int_{J_\iota} \sum_{n\in\N} \sum_{j \in \N} \langle n+2j+s \rangle^{2v+|\beta^\iota|} \int_0^\infty \int_0^\infty  M^{4+2v-2|\alpha|+|\beta^\iota|} \\
&\quad\times \mu^{2b_\iota + 4r-2l_\iota} \, |F^{(2k_\iota-|\beta^\iota|+v+r)}(\rho \langle n+2j+s \rangle + \mu^2)|^2 \, \tilde\chi(\rho/M) \,d\mu \,d\rho \,d\nu_\iota(s) \\
&\leq C_{K,\chi,\alpha} \sum_{r=\lceil l_\iota/2\rceil}^{l_\iota} \sum_{v=0}^{|\gamma^\iota|} M^{4+2v-2|\alpha|+|\beta^\iota|} 
\int_{0}^\infty \int_{0}^\infty |F^{(2k_\iota-|\beta^\iota|+v+r)}(\rho + \mu^2)|^2 \\
&\quad\times \mu^{2b_\iota + 4r-2l_\iota} \int_{J_\iota} \sum_{(n,j)\in\N^2} \langle n+2j+s \rangle^{2v+|\beta^\iota|-1} \tilde\chi(\rho/(\langle n + 2j + s \rangle M)) \,d\nu_\iota(s) \,d\mu \,d\rho,
\end{split}\]
by passing to polar coordinates and rescaling. The sum in $(n,j)$ is dominated by $(\rho/M)^{2v+|\beta^\iota|+1}$, uniformly in $s \in J_\iota$, and moreover $\supp F \subseteq K$. Therefore the left-hand side of \eqref{eq:weightedsummand} is majorized by
\[\begin{split}
C_{K,\chi,\alpha} \, M^{3-2|\alpha|} \sum_{r=\lceil l_\iota/2\rceil}^{l_\iota} \sum_{v=0}^{|\gamma^\iota|} \sup_{u \in \leftclosedint 0,\max K\rightclosedint} \int_0^\infty |F^{(2k_\iota-|\beta^\iota|+v+r)}(\rho + u)|^2 \,d\rho.
\end{split}\]
On the other hand, $b_\iota + |\beta^\iota| = l_\iota + 2k_\iota$, hence $2k_\iota-|\beta^\iota|+v+r \leq 2k_\iota-|\beta^\iota| + |\gamma^\iota| + l_\iota = b_\iota + |\gamma^\iota| \leq |\alpha|$ if $r$ and $v$ are in the range of summation, therefore the last integral is dominated by $\|F\|_{W_2^{|\alpha|}}$ uniformly in $r,v,u$, and \eqref{eq:weightedsummand} follows.
\end{proof}

\begin{prp}\label{prp:weightedl2}
Let $F : \R \to \C$ be smooth and such that $\supp F \subseteq K$ for some compact set $K \subseteq \leftopenint 0, \infty \rightopenint$. For all $r \in \leftclosedint 0, 3/2 \rightopenint$,
\[
\int_{N_{3,2}} \left| (1+|y|)^r \, \Kern_{F(L)}(x,y) \right|^2 \,dx \,dy \leq C_{K,r} \|F\|_{W_2^r}^2.
\]
\end{prp}
\begin{proof}
Take $\chi \in C^\infty_c(\leftopenint 0,\infty \rightopenint)$ such that $\supp \chi \subseteq \leftclosedint 1/2,2 \rightclosedint$ and $\sum_{k \in \Z} \chi(2^{-k} t) = 1$ for all $t \in \leftopenint 0,\infty \rightopenint$. Note that, if $(\lambda,\eta)$ belongs to the joint spectrum of $L,\vecY$, then $|\eta| \leq \lambda$. Therefore, if $k_K \in \Z$ is sufficiently large so that $2^{k_K-1} > \max K$, and if $F_M$ is defined for all $M \in \leftopenint 0,\infty \rightopenint$ as in Lemma~\ref{lem:weightedplancherel}, then
\[F(L) = \sum_{k \in \Z , \, k \leq k_K} F_{2^{k}}(L,\vecY)\]
(with convergence in the strong sense). Hence an estimate for $\Kern_{F(L)}$ can be obtained, via Minkowski's inequality, by summing the corresponding estimates for $\Kern_{F_{2^{k}}}(L,\vecY)$ given by Lemma~\ref{lem:weightedplancherel}. If $r < 3/2$, then the series $\sum_{k \leq k_K} (2^{k})^{3/2-r}$ converges, thus
\[
\int_{N_{3,2}} \left| |y|^r \, \Kern_{F(L)}(x,y) \right|^2 \,dx \,dy \leq C_{K,r} \|F\|_{W_2^r}^2.
\]
The conclusion follows by combining the last inequality with the corresponding one for $r = 0$.
\end{proof}

Recall that $|\cdot|_\delta$ denotes a $\delta_t$-homogeneous norm on $N_{3,2}$, thus $|(x,y)|_\delta \sim |x|+|y|^{1/2}$. Interpolation then allows us to improve the standard weighed estimate for a homogeneous sublaplacian on a stratified group.

\begin{prp}\label{prp:improvedl2estimate}
Let $F : \R \to \C$ be smooth and such that $\supp F \subseteq K$ for some compact set $K \subseteq \leftopenint 0, \infty \rightopenint$. For all $r \in \leftclosedint 0, 3/2 \rightopenint$, $\alpha \geq 0$ and $\beta > \alpha + r$,
\begin{equation}\label{eq:improvedl2estimate}
\int_{N_{3,2}} \left| (1+|(x,y)|_\delta)^\alpha \, (1+|y|)^r \, \Kern_{F(L)}(x,y) \right|^2 \,dx \,dy \leq C_{K,\alpha,\beta,r} \|F\|_{W_2^\beta}^2.
\end{equation}
\end{prp}
\begin{proof}
Note that $1+|y| \leq C (1+|(x,y)|_\delta)^2$. Hence, in the case $\alpha \geq 0$, $\beta > \alpha + 2r$, the inequality \eqref{eq:improvedl2estimate} follows by the standard estimate \cite[Lemma 1.2]{mauceri_vectorvalued_1990}. On the other hand, if $\alpha = 0$ and $\beta \geq r$, then \eqref{eq:improvedl2estimate} is given by Proposition~\ref{prp:weightedl2}. The full range of $\alpha$ and $\beta$ is then obtained by interpolation (cf.\ the proof of \cite[Lemma 1.2]{mauceri_vectorvalued_1990}).
\end{proof}

We can finally prove the fundamental $L^1$-estimate, and consequently Theorem~\ref{thm:mhn32}.

\begin{proof}[Proof of Proposition~\ref{prp:l1estimate}]
Take $r \in \leftopenint 9/2-s,3/2 \rightopenint$. Then $s-r > 3/2 + 3-2r$, hence we can find $\alpha_1 > 3/2$ and $\alpha_2 > 3-2r$ such that $s - r > \alpha_1 + \alpha_2$. Therefore, by Proposition~\ref{prp:improvedl2estimate} and H\"older's inequality,
\[
\|\Kern_{F(L)} \|_1^2  \leq C_{k,s} \|F\|_{W_2^s}^2 \int_{N_{3,2}} (1+|(x,y)|_\delta)^{-2\alpha_1-2\alpha_2} \, (1+|y|)^{-2r} \,dx \,dy.
\]
The integral on the right-hand side is finite, because $2\alpha_1 > 3$, $\alpha_2 + 2r > 3$, and
\[(1+|(x,y)|_\delta)^{-2\alpha_1-2\alpha_2} \, (1+|y|)^{-2r} \leq C_s (1+|x|)^{-2\alpha_1} \, (1+|y|)^{-\alpha_2 - 2r},\]
and we are done.
\end{proof}

\bibliographystyle{amsabbrv}
\bibliography{n32}

\providecommand{\bysame}{\leavevmode\hbox to3em{\hrulefill}\thinspace}
\providecommand{\MR}{\relax\ifhmode\unskip\space\fi MR }
\providecommand{\MRhref}[2]{%
  \href{http://www.ams.org/mathscinet-getitem?mr=#1}{#2}
}
\providecommand{\href}[2]{#2}
\begin{thebibliography}{10}

\bibitem{astengo_hardys_2000}
F.~Astengo, M.~Cowling, B.~Di~Blasio, and M.~Sundari, \emph{Hardy's uncertainty
  principle on certain {L}ie groups}, J. London Math. Soc. (2) \textbf{62}
  (2000), no.~2, 461--472.

\bibitem{astengo_gelfand_2009}
F.~Astengo, B.~Di~Blasio, and F.~Ricci, \emph{Gelfand pairs on the {H}eisenberg
  group and {S}chwartz functions}, J. Funct. Anal. \textbf{256} (2009), no.~5,
  1565--1587.

\bibitem{christ_multipliers_1991}
M.~Christ, \emph{{$L^p$} bounds for spectral multipliers on nilpotent groups},
  Trans. Amer. Math. Soc. \textbf{328} (1991), no.~1, 73--81.

\bibitem{de_michele_heisenberg_1979}
L.~De~Michele and G.~Mauceri, \emph{{$L^{p}$} multipliers on the {H}eisenberg
  group}, Michigan Math. J. \textbf{26} (1979), no.~3, 361--371.

\bibitem{de_michele_mulipliers_1987}
\bysame, \emph{{$H^p$} multipliers on stratified groups}, Ann. Mat. Pura Appl.
  (4) \textbf{148} (1987), 353--366.

\bibitem{erdelyi_higher2_1981}
A.~Erd{\'e}lyi, W.~Magnus, F.~Oberhettinger, and F.~G. Tricomi, \emph{Higher
  transcendental functions. {V}ol. {II}}, Robert E. Krieger Publishing Co.
  Inc., Melbourne, Fla., 1981.

\bibitem{fischer_gelfand_2009}
V.~Fischer and F.~Ricci, \emph{Gelfand transforms of {${\rm SO}(3)$}-invariant
  {S}chwartz functions on the free group {$N_{3,2}$}}, Ann. Inst. Fourier
  (Grenoble) \textbf{59} (2009), no.~6, 2143--2168.

\bibitem{folland_hardy_1982}
G.~B. Folland and E.~M. Stein, \emph{Hardy spaces on homogeneous groups},
  Mathematical Notes, vol.~28, Princeton University Press, Princeton, N.J.,
  1982.

\bibitem{folland_harmonic_1989}
G.~B. Folland, \emph{Harmonic analysis in phase space}, Annals of Mathematics
  Studies, vol. 122, Princeton University Press, Princeton, NJ, 1989.

\bibitem{hebisch_multiplier_1993}
W.~Hebisch, \emph{Multiplier theorem on generalized {H}eisenberg groups},
  Colloq. Math. \textbf{65} (1993), no.~2, 231--239.

\bibitem{hebisch_multiplier_1995}
W.~Hebisch and J.~Zienkiewicz, \emph{Multiplier theorem on generalized
  {H}eisenberg groups. {II}}, Colloq. Math. \textbf{69} (1995), no.~1, 29--36.

\bibitem{martini_multipliers_2010}
A.~Martini, \emph{{Algebras of differential operators on Lie groups and
  spectral multipliers}}, {Tesi di perfezionamento (PhD thesis)}, Scuola
  Normale Superiore, Pisa, 2010, \texttt{arXiv:1007.1119}.

\bibitem{martini_spectral_2011}
A.~Martini, \emph{{Spectral theory for commutative algebras of differential
  operators on Lie groups}}, J. Funct. Anal. \textbf{260} (2011), no.~9,
  2767--2814.

\bibitem{martini_joint_2012}
\bysame, \emph{{Analysis of joint spectral multipliers on Lie groups of
  polynomial growth}}, Ann. Inst. Fourier (Grenoble) \textbf{62} (2012), no.~4,
  1215--1263.

\bibitem{martini_grushin2}
A.~Martini and D.~M\"uller, \emph{A sharp multiplier theorem for {G}rushin
  operators in arbitrary dimensions},  (2012), \texttt{arXiv:1210.3564}.

\bibitem{mauceri_vectorvalued_1990}
G.~Mauceri and S.~Meda, \emph{Vector-valued multipliers on stratified groups},
  Rev. Mat. Iberoamericana \textbf{6} (1990), no.~3-4, 141--154.

\bibitem{mller_spectral_1994}
D.~M{\"u}ller and E.~M. Stein, \emph{On spectral multipliers for {H}eisenberg
  and related groups}, J. Math. Pures Appl. (9) \textbf{73} (1994), no.~4,
  413--440.

\bibitem{mller_restriction_1990}
D.~M{\"u}ller, \emph{A restriction theorem for the {H}eisenberg group}, Ann. of
  Math. (2) \textbf{131} (1990), no.~3, 567--587.

\bibitem{mller_marcinkiewicz_1996}
D.~M{\"u}ller, F.~Ricci, and E.~M. Stein, \emph{Marcinkiewicz multipliers and
  multi-parameter structure on {H}eisenberg (-type) groups. {II}}, Math. Z.
  \textbf{221} (1996), no.~2, 267--291.

\bibitem{rothschild_hypoelliptic_1976}
L.~P. Rothschild and E.~M. Stein, \emph{Hypoelliptic differential operators and
  nilpotent groups}, Acta Math. \textbf{137} (1976), no.~3-4, 247--320.

\bibitem{thangavelu_lectures_1993}
S.~Thangavelu, \emph{Lectures on {H}ermite and {L}aguerre expansions},
  Mathematical Notes, vol.~42, Princeton University Press, Princeton, NJ, 1993.

\end{thebibliography}

\end{document}